\documentclass[12pt]{amsart}

\usepackage{amssymb}
\usepackage{amsmath,amscd}

\newtheorem{fact}{Fact}[section]
\newtheorem{lemma}[fact]{Lemma}
\newtheorem{theorem}[fact]{Theorem}
\newtheorem{definition}[fact]{Definition}
\newtheorem{example}[fact]{Example}
\newtheorem{rremark}[fact]{Remark}
\newenvironment{remark}{\begin{rremark} \rm}{\end{rremark}}
\newtheorem{proposition}[fact]{Proposition}
\newtheorem{corollary}[fact]{Corollary}

\DeclareMathOperator\C{\mathbb C}

\DeclareMathOperator\Z{\mathbb Z}

\DeclareMathOperator\Gr{Gr}
\DeclareMathOperator\Grkn{Gr_k(\C^n)}
\DeclareMathOperator\TGrkn{T^*\!\Gr_k(\C^n)}
\DeclareMathOperator\spa{span}
\DeclareMathOperator\Loc{Loc}
\DeclareMathOperator\hor{hor}
\DeclareMathOperator\ver{ver}
\DeclareMathOperator\Sym{Sym}
\DeclareMathOperator\id{{id}}
\DeclareMathOperator\Id{Id}
\DeclareMathOperator\End{End}

\DeclareMathOperator\Det{Det}
\DeclareMathOperator\Stab{Stab}

\def\II{{\mathcal I}}
\def\Ik{{\II_k}}
\def\XX{{\mathcal X}}
\def\Ibar{\bar{I}}
\def\zz{{\boldsymbol z}}
\def\CCn{(\C^2)^{\otimes n}}
\def\Czy{\C[\zz,y]}
\def\fratop{\genfrac{}{}{0pt}1}
\def\satop#1#2{\fratop{\scriptstyle#1}{\scriptstyle#2}}
\def\tt{{\boldsymbol t}}
\def\Wt{\tilde{W}}
\def\ox{\otimes}
\def\la{\lambda}
\def\B{{\mathcal B}}
\def\L{{\mathcal L}}
\def\si{\sigma}
\def\Imin{I^{min}}
\def\Imax{I^{max}}
\def\R{{\mathcal R}}
\def\ts{\tilde{s}}
\def\V{{\CCn\ox \L_{\la,(\zz,y)}}}
\def\ep{\epsilon}

\newcommand{\slt}{{\frak{sl}_2}}
\newcommand{\es}{{E_y(\frak{gl}_2)}}

\def\h{{\mathfrak h}}

\let\ga\gamma
\let\Ga\Gamma
\let\al\alpha

\let\dl\delta
\let\on\operatorname
\let\bi\bibitem

\newcommand{\Ref}[1]{{\rm(\ref{#1})}}
\newcommand{\bea}{\begin{eqnarray*}}
\newcommand{\eea}{\end{eqnarray*}}
\newcommand{\bean}{\begin{eqnarray}}
\newcommand{\eean}{\end{eqnarray}}

\author
[R.\,Rim\'anyi,  A.\,Varchenko]
{ R.\,Rim\'anyi$\,^{\star}$,
 A.\,Varchenko$\,^\diamond$}

\title[Dynamical Gelfand-Zetlin algebra  and equivariant cohomology ]{ Dynamical Gelfand-Zetlin algebra and equivariant cohomology of Grassmannians}

\begin{document}

\maketitle

\begin{center}
{\it Department of Mathematics, University
of North Carolina at Chapel Hill\\ Chapel Hill, NC 27599-3250, USA}
\end{center}

{\let\thefootnote\relax
\footnotetext{
\noindent
$^\star${ E-mail}: rimanyi@email.unc.edu, 
supported in part by NSF grant DMS-1200685,
\\
\phantom{aa}
$^\diamond$  {E-mail}:  anv@email.unc.edu,
supported in part by NSF grant DMS-1362924 and Simons Foundation.}}

\begin{abstract}
We consider the rational dynamical quantum group $\es$ and introduce an $\es$-module
structure on $\oplus_{k=0}^n H^*_{GL_n\times\C^\times}(\TGrkn)'$, where 
$ H^*_{GL_n\times\C^\times}(\TGrkn)'$ is the equivariant cohomology algebra $H^*_{GL_n\times\C^\times}(\TGrkn)$
of the cotangent bundle of the Grassmannian $\Gr_k(\C^n)$ with coefficients
extended by a suitable ring of rational functions in an additional variable $\la$. We consider the dynamical
Gelfand-Zetlin algebra which is a commutative algebra of difference operators in $\la$. We show that
the action of the Gelfand-Zetlin algebra on  $H^*_{GL_n\times\C^\times}(\TGrkn)'$ is the natural action of the algebra
$H^*_{GL_n\times\C^\times}(\TGrkn)\otimes \C[\delta^{\pm1}]$ on $H^*_{GL_n\times\C^\times}(\TGrkn)'$, where $\delta : \zeta(\la)\to
\zeta(\la+y)$ is the shift operator.

The $\es$-module structure on $\oplus_{k=0}^n H^*_{GL_n\times\C^\times}(\TGrkn)'$ is introduced with the help of dynamical stable envelope maps  which are dynamical analogs of the stable envelope maps introduced by Maulik and Okounkov, \cite{MO}. The dynamical stable envelope maps are defined in terms of the rational dynamical weight functions introduced in \cite{FTV} to construct q-hypergeometric solutions of rational qKZB equations. The cohomology classes in $ H^*_{GL_n\times\C^\times}(\TGrkn)'$  induced by the weight functions are dynamical variants of 
Chern-Schwartz-MacPherson classes of Schubert cells.

\end{abstract}

{\small \tableofcontents }

\section{Introduction}

In \cite{MO}, Maulik and Okounkov study the classical and quantum equivariant
cohomology of Nakajima quiver varieties for a quiver $Q$. They construct a Hopf algebra
$Y_Q$, called the Yangian of $Q$, acting on the cohomology of these
varieties, and show that the  Bethe algebra $\B^q$ of this action, depending on some parameters $q$ and acting on the cohomology of these
varieties coincides with the algebra of quantum multiplication.
 If $q\to \infty$, the limiting Bethe algebra $\B^\infty$,  called the Gelfand-Zetlin algebra, is isomorphic to the algebra of the standard multiplication on the cohomology.
The construction of the Yangian and the Yangian action is based on the notion
of the stable envelope maps introduced in \cite{MO}.
In this paper we construct the dynamical  analog of the stable envelope maps for the
equivariant cohomology algebras of the cotangent bundles of  Grassmannians.

Let $(\C^\times)^n\subset GL_n$ be the torus of diagonal matrices. The groups $(\C^\times)^n\subset GL_n$ act
on $\C^n$ and hence on the cotangent bundle $\TGrkn$ of a Grassmannian.   Extend these $(\C^\times)^n\subset GL_n$ actions to the actions of
$T=(\C^\times)^n\times\C^\times\subset GL_n\times\C^\times$ in such a way that the extra $\C^\times$ acts on the fibers of $\TGrkn\to\Grkn$ by multiplication.
 Consider the equivariant cohomology algebra
$H^*_T(\TGrkn)$. In this situation Maulik and Okounkov define the stable envelope maps
\bea
\Stab_\si : \oplus_{k=0}^n
H^*_T((\TGrkn)^T)
\to
\oplus_{k=0}^n
H^*_T(\TGrkn),
\eea
where $(\TGrkn)^T\subset \TGrkn$ is the fixed point set with respect to the action of $T$ and $\si$ is an element of the symmetric group $S_n$.
They describe the composition maps $\Stab^{-1}_{\:\si'}\circ\Stab_{\:\si}$
in terms of the standard $\frak{gl}_2$ rational \,R-matrix and these R-matrices give 
a Yangian $Y(\frak{gl}_2)$-module structure on $\oplus_{k=0}^n
H^*_T(\TGrkn)$, see also \cite{RTV1}.

A similar construction can be performed  for the equivariant K-theory algebras $K_T(\TGrkn)$, see \cite{RTV2}. In that case
the composition maps $\Stab^{-1}_{\:\si'}\circ\Stab_{\:\si}$ are described
in terms of the standard $\frak{gl}_2$ trigonometric  \,R-matrix and these R-matrices give 
a  $\frak{gl}_2$ quantum loop algebra action on 
 $\oplus_{k=0}^n K_T((\TGrkn)^T)$ similar to the quantum loop algebra action
studied by Ginzburg and Vasserot in \cite{GV,Vas1,Vas2}.  

Let $\la$ be a new variable. In this paper we extend the coefficients $H^*_T((\TGrkn)^T) \subset H^*_T((\TGrkn)^T)'$ and
 $H^*_T(\TGrkn) \subset H^*_T(\TGrkn)'$ by suitable rational functions in $\la$ and then define dynamical stable envelope maps 
\bean
\label{Dyn}
\Stab_\si : \oplus_{k=0}^n
H^*_T((\TGrkn)^T)'
\to
\oplus_{k=0}^n
H^*_T(\TGrkn))', \qquad \si\in S_n.
\eean
We describe the composition maps $\Stab^{-1}_{\:\si'}\circ\Stab_{\:\si}$ 
in terms of the rational dynamical R-matrix 
\bea
\label{RM}
R(\la,z,y)=
\left( \begin{array}{cccc}
1 & 0 & 0 & 0\\
0 & \frac{(\lambda + y)z}{\lambda(z-y)} &  -\frac{(\lambda + z)y}{\lambda(z-y)} &0
\\
0 & -\frac{(\lambda - z)y}{\lambda(z-y)}  & \frac{(\lambda - y)z}{\lambda(z-y)}&0
\\
0 & 0 & 0 & 1
\end{array} \right)
\eea
and define on 
 $\oplus_{k=0}^n H^*_T(\TGrkn)'$ a module structure over the rational dynamical quantum group
 $\es$.
 
 The elliptic dynamical quantum group $E_{\tau,\eta}(\frak{gl}_2)$ was introduced by G.\,Felder in \cite{F1, F2}, see
also \cite{FV1}-\cite{FV3}, \cite{FTV}. The rational dynamical group $\es$ is a suitable semi-classical limit $\tau \to i\infty$ of 
$E_{\tau,\eta}(\frak{gl}_2)$, where $\tau$ is the modular parameter.

Having the $\es$-module structure on $\oplus_{k=0}^n H^*_T(\TGrkn)'$ we consider the Gelfand-Zetlin algebra $\B$ of that module,
the dynamical analog of the Gelfand-Zetlin subalgebra of the Yangian $Y(\frak{gl}_2)$. The  Gelfand-Zetlin algebra $\B$ is a 
commutative algebra of  difference operators with respect to the variable $\la$. The Gelfand-Zetlin algebra preserves each of the terms
$H^*_T(\TGrkn)'$.  Let $\delta  : \zeta(\la) \to \zeta(\la +y)$ be the shift operator.
We show that the action of  $\B$ on $H^*_T(\TGrkn)'$ is the natural action of $H^*_{GL_n\times\C^\times}(\TGrkn)\otimes \C[\delta^{\pm1}]$
on $H^*_T(\TGrkn)'$, where 
$\C[\delta^{\pm1}]$ is the algebra of Laurent polynomials in $\delta$,  see Theorem \ref{thm MAIN}.

We show that the space $\oplus_{k=0}^n H^*_{GL_n\times\C^\times}(\TGrkn)'\subset \oplus_{k=0}^n H^*_T(\TGrkn)'$ is an $\es$-submodule and
show that for any $k$,
  the action of  $\B$ on $H^*_{GL_n\times\C^\times}(\TGrkn)'$ is the natural action of $H^*_{GL_n\times\C^\times}(\TGrkn)\otimes \C[\delta^{\pm1}]$
on $H^*_{GL_n\times\C^\times}(\TGrkn)'$, see Corollary \ref{cor last}.

\medskip
We define  the dynamical stable envelope maps for the cotangent bundles of Grassmannians by explicit formulas 
in terms of the $\frak{gl}_2$ rational dynamical weight functions introduced in \cite{FTV}.   Our motivation of this construction is the fact that  the stable envelope maps for the equivariant cohomology (resp. K-theory) of cotangent  bundles of
 Grassmannians were defined in terms of the rational (resp. trigonometric) $\frak{gl}_2$ weight functions from \cite{TV1}, and our
 goal was to study the stable envelope maps in the dynamical setting.

The rational dynamical weight functions were introduced in \cite{FTV}
to construct $q$-hypergeometric solutions of the rational qKZB equations. The arguments of the weight functions in
\cite{FTV} are $\la$,  $y$, $z_1,\dots,z_n$, $t_1,\dots,t_k$, where $\la$ runs through the one-dimensional Cartan subalgebra of $\frak{sl}_2$,
$y$ is the parameter of the dynamical quantum group $\es$, variables
$z_1,\dots,z_n$ are evaluation parameters of tensor factors
 and $t_1,\dots,t_k$ are the integration variables in
the $q$-hypergeometric integrals. Another interpretation in \cite{FTV} of  variables $t_1,\dots,t_k$ 
is that they are variables in
the Bethe ansatz equations associated with the dynamical XXX model. In this paper, variable $\la$ is an auxiliary parameter, which could be interpreted as the parameter corresponding to the hyperplane sections of Grassmannians, variables  $y$, $z_1,\dots,z_n$ are interpreted as the equivariant parameters corresponding
to the torus $T$ and the arguments \,$t_1,\dots,t_k$ are interpreted as the Chern
roots of the associated bundle over $\Grkn$. This correspondence between the
variables in the Bethe ansatz equations and the Chern roots is the indication of
a dynamical Landau-Ginzburg mirror correspondence.

In \cite{MO} in the case of Grassmannians, the commutative Bethe algebra $\B^q$ depends on quantum parameter $q$ which also corresponds to the hyperplane sections of Grassmannians. The  limit $q\to\infty$ of $\B^q$ gives the Gelfand-Zetlin algebra.
In  the dynamical setting our Gelfand-Zetlin algebra  does not have obvious deformations  and it is not clear  if there is a quantum parameter. One may speculate that the dynamical parameter $\la$ is an analog of the quantum parameter $q$ in  quantum cohomology.
 
Notice that the dynamical quantum group has one more nontrivial commutative
Bethe algebra $\tilde L_{11}(w)+\tilde L_{22}(w)$, studied in \cite{F1, F2, FV1, FV2, FV3, FTV}.

In the next paper we will extend the constructions and results of this paper to the case of the elliptic dynamical quantum group
 $E_{\tau,\,\eta}(\frak{gl}_2)$. That will give us an elliptic dynamical version of the $H^*_{GL_n\times\C^\times}(\TGrkn)\otimes \C[\delta^{\pm1}]$-module
  $H^*_{GL_n\times\C^\times}(\TGrkn)'$.  
 
 This paper can be considered as a continuation of the series of papers \cite{RSTV, GRTV, RTV1, RTV2, TV3} devoted to the geometrization  of the Bethe algebras in quantum integrable models.

\medskip
The paper is organized as follows.  In Section \ref{Preliminaries from Geometry} we collect geometric information on Grassmannians.
In Section \ref{sec:weightfunctions} we introduce the rational dynamical weight functions 
$\{W_{\si,I}\}$  and establish their recursion  and
orthogonality properties. In Section \ref{combI} we discuss the
diagrammatic interpretation of the combinatorics encoded in
the weight functions. In Section \ref{sec:interpolation} the interpolation properties of the weight functions are collected and important   classes $\{\kappa_{\si,I}\}\in \oplus_{k=0}^n H^*_T(\TGrkn)'$ are introduced.
The classes  $\{\kappa_{\sigma,I}\}$   are  dynamical deformations of 
the Chern-Schwartz-MacPherson classes of Schubert cells, see Remark \ref{rem CSM} and \cite{RV}.
In Section \ref{Stable envelope maps and R-matrices} we define the dynamical stable envelope maps
\[
\Stab_{\sigma}:
\oplus_{k=0}^n
H^*_T((\TGrkn)^T)'
\to
\oplus_{k=0}^n
H^*_T(\TGrkn))', \qquad 
1_I \mapsto \kappa_{\sigma,I},
\]
and calculate the composition maps $\Stab^{-1}_{\:\si'}\circ\Stab_{\:\si}$
in terms of the  dynamical R-matrix $R(\la,z,y)$.
In Sections \ref{sec:xi} and  \ref{sec:inv} we introduce a collection of elements $\{\xi_I\}_{I\in\Ik} \in H^*_T((\TGrkn)^T)'$
 and construct the map inverse to the map $\Stab_{\id}$.
In Section \ref{sec dyn g} we present information about the rational dynamical quantum group $\es$ and introduce the Gelfand-Zetlin algebra.
We show that the action of the Gelfand-Zetlin algebra on $H^*_T((\TGrkn)^T)'$ is diagonalizable.
Proposition \ref{eigen} says that the elements $\{\xi_I\}_{I\in\Ik}$ form a basis of eigenfunctions of the Gelfand-Zetlin algebra.
In Section \ref{on coho} we prove the main results:\, Theorem \ref{thm MAIN} and Corollary \ref{cor last}.

\medskip
The authors thank Giovanni Felder for useful discussions.
The second author thanks MPI in Bonn for hospitality.

\section{Preliminaries from Geometry}\label{sec:geo}
\label{Preliminaries from Geometry}

\subsection{Grasmannians}
\label{sec Partial flag varieties}

Fix a natural number $n$. For $k\leq n$ consider the Grassmannian $\Gr_k(\C^n)$ of linear subspaces of $\C^n$, and its cotangent bundle $\pi:\TGrkn \to \Grkn$. Let
\[
\XX_n=\coprod_{k=0}^n \TGrkn.
\]
The set of $k$-element subsets of $[n]:=\{1,\ldots,n\}$ will be denoted by $\Ik$. For $I\in \Ik$ let $\Ibar=[n]-I$.

Let $\ep_1,\ldots,\ep_n$ be the standard basis of $\C^n$. For $I\in \Ik$ let $x_I$ be the point in $\Grkn$ corresponding to the coordinate subspace $\spa\{\ep_i\ |\ i\in I\}$. We embed $\Grkn$ in $\TGrkn$ as the zero section and consider the points $x_I$ as points of $\TGrkn$.

\subsection{Schubert cells, conormal bundles}
For any $\sigma\in S_n$ we consider the full coordinate flag in $\C^n$
\[
V^\sigma: \ \ 0\subset V^\sigma_1 \subset V^\sigma_2 \subset \ldots \subset V^\sigma_n=\C^n,
\]
where $V^\sigma_i=\spa\{ \ep_{\sigma(1)},\ldots,\ep_{\sigma(i)}\}$. For $I\in \Ik$ define the Schubert cell
\[
\Omega_{\sigma,I}=\{ F\in \Grkn\ |\  \dim(F\cap V^\sigma_q)=\#\{i\in I\ |\ \sigma^{-1}(i)\leq q\}\ \text{for}\ q=1,\ldots,n\}.
\]
The Schubert cell is isomorphic to an affine space of dimension
\[
\ell_{\sigma,I}=\#\{(i,j)\in [n]^2 \ |\  i>j, \sigma(i)\in I, \sigma(j)\in \Ibar\}.
\]
For a fixed $\sigma$, the Grassmannian $\Grkn$ is the disjoint union of the cells $\Omega_{\sigma,I}$, see e.g. \cite[Sect.2.2]{FP}. We have $x_I\in \Omega_{\sigma,I}$ for any $\sigma$.

For $\sigma\in S_n$ we define the partial ordering on the set $\Ik$. 
For $I,J\in\Ik$, let
\[
\sigma^{-1}(I)=\{i_1 < \ldots < i_k\},
\qquad
\sigma^{-1}(J)=\{j_1 < \ldots < j_k\}.
\]
We say that $J\leq_{\sigma} I$ if $j_a\leq i_a$ for $a=1,\ldots, k$.

The Schubert cell $\Omega_{\sigma,I}$ is a smooth submanifold of $\Grkn$. Consider its conormal space
\[
C\Omega_{\sigma,I}=\{\alpha \in \pi^{-1}(\Omega_{\sigma,I}) \ | \
\alpha(T_{\pi(\alpha)}\Omega_{\sigma,I})=0\} \subset\TGrkn.
\]
It is  the total space of a vector subbundle of $\TGrkn$ over $\Omega_{\sigma,I}$.

\subsection{Equivariant cohomology}
\label{sec:equivH}

Let $(\C^\times)^n\subset GL_n=GL_n(\C)$ be the torus of diagonal matrices. The groups $(\C^\times)^n\subset GL_n$ act
on $\C^n$ and hence on the cotangent bundle $\TGrkn$.   We extend these $(\C^\times)^n\subset GL_n$ actions to the actions of
$T=(\C^\times)^n\times\C^\times\subset GL_n\times\C^\times$ in such a way that the extra $\C^\times$ acts on the fibers of $\TGrkn\to\Grkn$ by multiplication.

Consider the set of variables
\[
\Gamma=
\{\gamma_{1,1}, \gamma_{1,2},\ldots,\gamma_{1,k} , \gamma_{2,1},\gamma_{2,2},\ldots,\gamma_{2,n-k}\}.
\]
The group $S_k\times S_{n-k}$ acts on $\Gamma$ by permuting the $\gamma_{1,*}$ and $\gamma_{2,*}$ variables independently. Consider also variables $\zz=\{z_1,\ldots, z_n\}$ and $y$.
The group $S_n$ acts on the set $\zz$ by permutations.
The following presentation of the equivariant cohomology of the Grassmannian is well known.
We have
\begin{equation}
\label{eqn:Hpresentation}
H^*_T(\TGrkn)=\C[\Gamma,\zz,y]^{S_k\times S_{n-k}}/
\prod_{i=1}^k (u-\gamma_{1,i}) \prod_{i=1}^{n-k} (u-\gamma_{2,i}) -
\prod_{i=1}^n (u-z_i).
\end{equation}
Here the meaning of $\prod_{i=1}^k (u-\gamma_{1,i}) \prod_{i=1}^{n-k} (u-\gamma_{2,i}) -
\prod_{i=1}^n (u-z_i)$ is that the coefficient of every power of $u$ in this expression is set to be a relation. In other words, the relations in this algebra are of the form $f(\Gamma)=f(\zz)$ for symmetric polynomials $f$.
In this description the variables $\gamma_{1,*}$ (resp. $\gamma_{2,*}$) are the Chern roots of the tautological subbundle
(resp. quotient bundle) over $\Grkn$. The variables $z_i$ are the Chern roots corresponding to the factors of $T^n$.

Similarly we have the representation
\begin{equation}
\label{eqn:Hpr}
H^*_{GL_n\times\C^\times}(\TGrkn)=\C[\Gamma,\zz,y]^{S_k\times S_{n-k}\times S_n}/
\prod_{i=1}^k (u-\gamma_{1,i}) \prod_{i=1}^{n-k} (u-\gamma_{2,i}) -
\prod_{i=1}^n (u-z_i).
\end{equation}
We have the natural embedding
\bean
\label{Emb}
H^*_{GL_n\times\C^\times}(\TGrkn)\subset H^*_{T}(\TGrkn).
\eean

\subsection{Denominators} \label{sec:denom}

The cohomology algebra $H^*_T(\TGrkn)$
is a module over the polynomial algebra $\C[\zz,y]$. Later in this paper this and other $\C[\zz,y]$-modules will be considered with some permitted denominators, as follows.

Let $\L_{\zz,y}$ be the algebra consisting of rational functions of the form
\begin{equation}
\frac{f(\zz,y)}
{\prod_{i \not= j} (z_i-z_j)^{k_{ij}} \prod_{i,j} (z_i-z_j-y)^{l_{ij}} }
\end{equation}
where $f$ is a polynomial and $k_{ij}, l_{ij}$ are arbitrary nonnegative integers.

Let $\L_{(\zz,y)}$ be the algebra of all rational functions in $\zz, y$.

Let $\lambda$ be a new variable. Let $\L_{\la,y}$ be the algebra of rational functions
of the form $f/g$, where $f$ is a polynomial
in variables $\la, y$ and $g$ is a finite product of factors of the form $\la + ly$, where $l\in\Z$.

Let
\bean
\label{rings}
\L_{\lambda,\zz,y}=\L_{\lambda,y} \ox_{\C[y]} \L_{\zz,y},
\qquad
\L_{\lambda, (\zz,y)}=\L_{\lambda,y} \ox_{\C[y]} \L_{(\zz,y)}.
\eean

For a module $M$ over $\C[\zz,y]$ by $M\ox \L_{\zz,y}$,  $M\ox \L_{\la,\zz,y}$ or  $M\ox \L_{\la,(\zz,y)}$
 we mean the tensor products over $\C[\zz,y]$.

\subsection{Fixed point sets} \label{sec:fixed}

The set $(\TGrkn)^T$ of fixed points of the $T$-action is $(x_I)_{I\in\Ik}$. We have
\[
(\XX_n)^T = \XX_1 \times \ldots \times \XX_1, \qquad \text{($n$ factors)}.
\]
Let the standard basis vectors of $\C^2$ be $v_1$ and $v_2$. The algebra $H^*_T\bigl((\XX_n)^T\bigr)$ is naturally isomorphic to $\CCn \otimes \Czy$, with the isomorphism mapping the identity element $1_I\in H^*_T(x_I)$ to the vector
\[\label{v_I}
v_I=v_{i_1}\otimes \ldots \otimes v_{i_n},
\]
where $i_j=1$ if $j\in I$ and $i_j=2$ if $j\in \Ibar$. We denote by $\CCn_k$ the span of $\{v_I \ |\ |I|=k\}$.

\subsection{Equivariant localization} \label{sec:loc}

Consider the {\em equivariant localization} map
\[\label{eqn:loc}
\Loc: H^*_T(\TGrkn)\to H^*_T((\TGrkn)^T)= \bigoplus_{I\in\Ik} H^*_T(x_I)
\]
whose components are the restrictions to the fixed points $x_I$.
In the description (\ref{eqn:Hpresentation}), the $I$-component $\Loc_I$ of this map is the substitution
\[\label{LocI}
\{\gamma_{1,*}\}\mapsto\{z_a\ |\ a\in I\},
\qquad
\{\gamma_{2,*}\}\mapsto\{z_a\ |\ a\in \Ibar\}.
\]
For  $f(\Gamma,\zz,y) \in H^*_T(\TGrkn)$ the result of this substitution will be denoted by $f(\zz_I,\zz,y)$.

Equivariant localization theory (see e.g. \cite{ChG}) asserts that $\Loc$ is an injection of algebras.
Moreover, an element of the right-hand side is in the image of $\Loc$ if the difference
of the $I$-th and $s_{i,j}(I)$-th components is divisible by $z_i-z_j$ in $\Czy$ for any $i\ne j$.
Here $s_{i,j}\in S_n$ is the transposition of $i$ and $j$, and 
 $s_{i,j}(I)$ is the set obtained from $I$ by switching the numbers $i$ and $j$.
Hence the maps
\begin{equation}
\label{eqn:loc2}
\Loc: H^*_T(\Grkn)\ox \L_{\zz,y} \rightarrow \oplus_{I\in\Ik} H^*_T(x_I)\ox \L_{\zz,y},
\end{equation}
\[
\Loc: H^*_T(\Grkn)\ox \L_{\lambda,\zz,y} \rightarrow \oplus_{I\in\Ik} H^*_T(x_I)\ox \L_{\lambda,\zz,y}
\]
are isomorphisms.

\subsection{Fundamental class of $\Omega_{\sigma,I}$ at $x_I$} \label{sec:edefs}

Define the following classes
\[\label{eqn:ehor}
e_{\sigma,I,+}^{\hor}=\prod_{b<a}
\prod_{\satop{\sigma(a)\in I}{\sigma(b)\in \Ibar}} (z_{\sigma(b)}-z_{\sigma(a)}),
\qquad
e_{\sigma,I,-}^{\hor}=\prod_{b>a}
\prod_{\satop{\sigma(a)\in I}{\sigma(b)\in \Ibar}} (z_{\sigma(b)}-z_{\sigma(a)}),
\]
\[
e_{\sigma,I,+}^{\ver}=\prod_{b>a}
\prod_{\satop{\sigma(a)\in I}{\sigma(b)\in \Ibar}} (z_{\sigma(a)}-z_{\sigma(b)}+y),
\qquad
e_{\sigma,I,-}^{\ver}=\prod_{b<a}
\prod_{\satop{\sigma(a)\in I}{\sigma(b)\in \Ibar}} (z_{\sigma(a)}-z_{\sigma(b)}+y)
\]
in $H^*_T($pt$)=\Czy$. Here $hor$ and $ver$ refer to horizontal and vertical.

Recall that if $\C^\times$ acts on a line $\C$ by $\alpha\cdot x=\alpha^r x$ then the $\C^\times$-equivariant Euler class of the line bundle $\C \to \{0\}$ is $e(\C\to \{0\})=rz \in H^*_{\C^\times}($point$)=\C[z]$. Thus standard knowledge on the tangent bundle of Grassmannians imply that
\[\label{eqn:einter1}
e(T\Omega_{\sigma,I}|_{x_I})=e_{\sigma,I,+}^{\hor}, \qquad
e(\nu(\Omega_{\sigma,I}\subset\Grkn)|_{x_I})=e_{\sigma,I,-}^{\hor},
\]
where $\nu(A\subset B)$ means the normal bundle of a submanifold $A$ in the ambient manifold $B$, and $\xi|_x$ means the restriction of the bundle $\xi$ over the point $x$ in the base space. Therefore we also have
\[ e(C\Omega_{\sigma,I}|_{x_I})=e_{\sigma,I,+}^{\ver},\qquad
e((\pi^{-1}(\Omega_{\sigma,I})-C\Omega_{\sigma,I})|_{x_I})=e_{\sigma,I,-}^{\ver},
\]
where $C\Omega_{\sigma,I}$ and $\pi^{-1}(\Omega_{\sigma,I})$ are considered bundles over $\Omega_{\sigma,I}$. Now consider $C\Omega_{\sigma,I}$ as a submanifold of $\TGrkn$. Then we obtain
\[\label{eqn:enu}
e(\nu( C\Omega_{\sigma,I} \subset\TGrkn)|_{x_I})=
e^{\hor}_{\sigma,I,-}e^{\ver}_{\sigma,I,-}.
\]

\section{Weight functions} \label{sec:weightfunctions}

\subsection{Definition and recursion of weight functions}

\begin{definition}
For $i\in I \in \Ik$ let
\[
w(i,I)=-\#\{j\in [n] \ |\  j>i, j\in \Ibar\}     + \#\{j\in [n] \ |\ j>i, j\in I\},
\]
\end{definition}

The notation does not record it, but $w(i,I)$ depends on $n$ as well. For example for $n=6$ we have $w(2,\{1,2,4\})=-3+1=-2$, and for $n=7$ we have $w(2,\{1,2,4\})=-4+1=-3$.

\begin{definition}\label{def:polarization} For $I\in\Ik, \sigma\in S_n$, $r\in \Z$ define the rational functions
\[
C^{(r)}_{\sigma,I}=\prod_{i\in I} (\lambda -(w(\sigma^{-1}(i),\sigma^{-1}(I))+r)y).
\]
\end{definition}

Let $I=\{i_1<i_2<\ldots<i_k\}\subset [n]$. For $r=1,\ldots,k$, $a=1,\ldots,n$ let
\[
  l_I(r,a) = \begin{cases}
   \ \hskip .62 true cm t_r-z_a+y                           & \text{if } a< i_r \\
   \lambda+t_r-z_a-w(i_r,I)y           \hskip .5 true cm    & \text{if } a=i_r \\
   \ \hskip .62 true cm t_r-z_a                             & \text{if } a> i_r.  \end{cases}
\]

For $\sigma\in S_n$ and $I=\{i_1,\ldots,i_k\}\subset [n]$ define $\sigma(I)=\{\sigma(i_1),\ldots,\sigma(i_k)\}\subset [n]$. We will define some functions in the variables $\lambda$, $\tt=(t_1,\ldots,t_k)$, $\zz=(z_1,\ldots,z_n)$, and $y$.

\begin{definition}
Let
\[
W_I(\lambda,\tt,\zz,y)={y^k} \cdot \Sym_{k} \left( \prod_{r=1}^k \prod_{a=1}^n l_I(r,a) \cdot \prod_{a=1}^{k-1} \prod_{b=a+1}^k \frac{t_a-t_b+y}{t_a-t_b}\right),
\]
where $\Sym_{k} f$ means $\sum_{\sigma\in S_k} f(t_{\sigma(1)},\ldots,t_{\sigma(k)})$.
Define the (dynamical rational) weight function by
\[
W_{\sigma,I}(\lambda,\tt,\zz,y)=W_{\sigma^{-1}(I)}(\lambda,\tt,z_{\sigma(1)},\ldots,z_{\sigma(n)},y).
\]
\end{definition}

We have $W_{\id,I}=W_I$. Despite the appearance of $t_a-t_b$ factors in the denominators, $W_I$ and $W_{\sigma,I}$ are polynomials.

\begin{example}
For $n=2$ we have
\begin{align*}
W_{\id,\{1\}}&=y (\lambda+t_1-z_1+y)(t_1-z_2),
&
W_{\id,\{2\}}&=y (t_1-z_1+y)(\lambda+t_1-z_2), \\
W_{s_{1,2},\{1\}}&=y (\lambda+t_1-z_1)(t_1-z_2+y),
&
W_{s_{1,2},\{2\}}&=y (t_1-z_1)(\lambda+t_1-z_2+y).
\end{align*}
More generally, for $k=1$, and $\{i\}\subset [n]$ we have
\[
W_{\id, \{i\}}=
y \prod_{a=1}^{i-1}(t_1-z_a+y) \cdot (\lambda+t_1-z_i+(n-i)y) \cdot \prod_{a=i+1}^n (t_1-z_a).
\]
\end{example}

\def\Ii{I^i}

\begin{proposition} [Recursion for weight functions]
\label{prop:Wrec}

 If $s_{a,a+1}(I)=I$ then
\[
W_I(\ldots,z_{a+1},z_a,\ldots)=W_I(\ldots,z_a,z_{a+1},\ldots)
\]
If $a\in I$, $a+1\not\in I$ then
\begin{multline*}
W_{s_{a,a+1}(I)}(\ldots,z_{a+1},z_a,\ldots)=\\
\frac{(z_{a+1}-z_a)(\lambda-(w(a,I)+2)y)}{(z_{a+1}-z_a+y)(\lambda-(w(a,I)+1)y)}
W_I +
\frac{y(\lambda+z_{a+1}-z_a-(w(a,I)+1)y)}{(z_{a+1}-z_a+y)(\lambda-(w(a,I)+1)y)}
W_{s_{a,a+1}(I)}.
\end{multline*}
If $a\not\in I$, $a+1\in I$ then
\begin{multline*}
W_{s_{a,a+1}(I)}(\ldots,z_{a+1},z_a,\ldots)=\\
\frac{(z_{a+1}-z_a)(\lambda-(w(a+1,I)-1)y)}{(z_{a+1}-z_a+y)(\lambda-w(a+1,I)y)}
W_I +
\frac{y(\lambda+z_{a+1}-z_a-w(a+1,I)y)}{(z_{a+1}-z_a+y)(\lambda-w(a+1,I)y)}
W_{s_{a,a+1}(I)}.
\end{multline*}
\end{proposition}

\begin{proof} Straightforward calculation, cf. \cite[Lemma 3.3]{RTV1}, \cite[Theorem 6.10]{RTV2}. The statement of this lemma is
the rational degeneration of the statement of Lemma 6 in \cite{FTV}.
\end{proof}

\subsection{Some versions of weight functions}\label{sec:Wversions}

Denote
\bea
e_k=e_k(\tt,y)=\prod_{a=1}^k\prod_{b=1}^k (t_a-t_b+y) .
\eea
Let $s_0=(n,n-1,\ldots,2,1)$ be the longest permutation in $S_n$.
We define the following modifications of weight functions:
\begin{align*}
\Wt_{\sigma,I}=\Wt_{\sigma,I}(\lambda,\tt,\zz,y)&=\frac1{e_k}W_{\sigma,I}(\lambda,\tt,\zz,y), \\
\Wt^-_K=\Wt^-_K(\lambda,\tt,\zz,y)&=\frac{(-1)^k}{e_k} W_{s_0,K}(-\lambda-(n-2k)y,\tt,\zz,y),\\
\Wt^+_{\sigma,J}=\Wt^+_{\sigma,J}(\lambda,\tt,\zz,y) &=\frac{1}{C^{(0)}_{\sigma,J}C^{(1)}_{\sigma,J}e_k}W_{\sigma,J}(\lambda,\tt,\zz,y),\\
\Wt^+_{J}=\Wt^+_{J}(\lambda,\tt,\zz,y)&= \Wt^+_{\id,J}.
\end{align*}

\begin{example}
For $k=1$ we have
\[
\Wt^-_{\{i\}}=
-\prod_{a=1}^{i-1}(t_1-z_a) \cdot (-\lambda+t_1-z_i+(i-n+1)y) \cdot \prod_{a=i+1}^n (t_1-z_a+y),
\]

\[
\Wt^+_{\{i\}}=\frac{
\prod_{a=1}^{i-1}(t_1-z_a+y) \cdot (\lambda+t_1-z_i+(n-i)y) \cdot \prod_{a=i+1}^n (t_1-z_a)}
{(\lambda+(n-i)y)(\lambda+(n-i-1)y)}.
\]
\end{example}

\subsection{Orthogonality}

Denote
\bean
\label{RQ}
R_I=R_I(\zz)=\prod_{a\in I}\prod_{b\in \Ibar} (z_a-z_b),\qquad
Q_I=Q_I(\zz,y)=\prod_{a\in I}\prod_{b\in \Ibar} (z_a-z_b+y).
\eean

\begin{proposition}[Orthogonality of weight functions I] \label{thm:ortho} For any $J$ and $K$ we have
\[
\sum_{I\in \Ik} \frac{
W_{\id,J}(\lambda,\zz_I,\zz,y)W_{s_0,K}(-\lambda-(n-2k)y,\zz_I,\zz,y)}
{
C_J(\lambda,y) e_k^2(\zz_I) R_I(\zz)Q_I(\zz,y)}
=(-1)^k\delta_{J,K}.
\]
\end{proposition}

\begin{proof}  The proof is as in \cite[Lemma 3.4]{RTV1}  and  \cite[Theorem 6.6]{RTV2}. The statement of this lemma is the rational degeneration of the statement of Theorem C.4 in \cite{TV2}.
\end{proof}

Define the scalar product
\[
(f(\lambda,\tt,\zz,y),g(\lambda,\tt,\zz,y)) = \sum_{I\in \Ik} \frac{ f(\lambda,\zz_I,\zz,y)g(\lambda,\zz_I,\zz,y) }{R_I(\zz)Q_I(\zz,y)}.
\]

\begin{corollary}[Orthogonality of weight functions II]\label{cor:ortho}
We have
\[
(\Wt^+_J,\Wt^-_K)=\delta_{J,K}.
\]
\end{corollary}

\begin{proof}
The statement is a rephrasing of Proposition \ref{thm:ortho}.
\end{proof}

\section{Combinatorics of the terms of the weight function}
\label{combI}

In this section we show a diagrammatic interpretation of the rich combinatorics
encoded in the weight function. Let $I\in \Ik$. Consider a table with $n$ rows
and two columns. Number the rows from top to bottom and number the columns from
left to right (by 1 and 2). Certain boxes of this table will be distinguished, as follows.
In the first column distinguish boxes in the $i$'th row if $i\in I$ and distinguish all the boxes in the second column.

Now we will define fillings of the tables by putting various variables in
the distinguished boxes. First, put the variables $z_1,\ldots, z_n$ into the second
column from top to bottom. Now choose a permutation
$\sigma\in S_k$ and put the variables
$t_{\si(1)},\ldots, t_{\si(k)}$ in the distinguished boxes of the first column from top to bottom.

Each such filled table will define a rational function as follows. Let $t_r$ be
a variable in the filled table in the first column. If $z_a$ is
a variable in the second column, but {\em above} the position of $t_r$ then consider
the factor $t_r-z_a+y$ (`type-1 factor'). If $z_a$ is a variable in the second column,
but {\em below} the position of $t_r$ then consider the factor $t_r-z_a$ (`type-2
factor'). If $z_a$ is a variable directly to the right of $t_r$ then consider the
factor $\lambda+t_r-z_a-wy$ where $w$ is the number of distinguished boxes below $t_r$ minus the number of non-distinguished boxes below $t_r$ (`type-3 factor').

Also, if $t_b$ is a variable in the first column, but {\em below} the position of $t_a$ then consider the factor $(t_a-t_b+y)/(t_a-t_b)$ (`type-4 factor').
The rule is illustrated in the following figure.
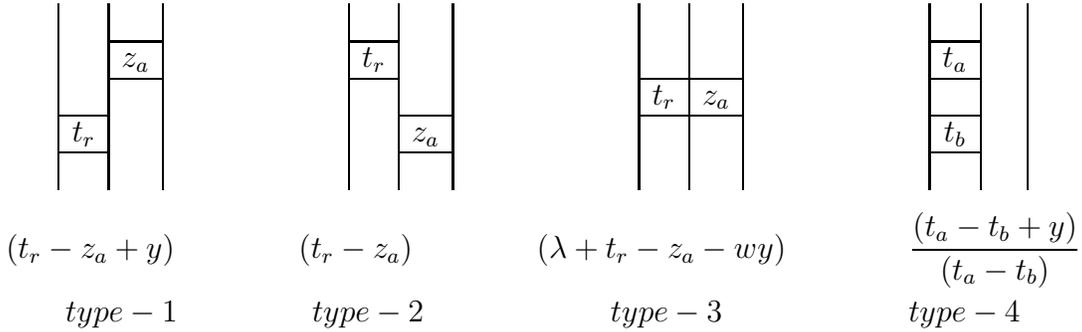
\[
\begin{tabular}{|c|c|}
& \\
\cline{2-2}
& \multicolumn{1}{|c|}{$z_a$} \\
\cline{2-2}
& \\
\cline{1-1}
\multicolumn{1}{|c|}{$t_r$} & \\
\cline{1-1}
&
\end{tabular} \qquad\qquad\qquad
\begin{tabular}{|c|c|}
& \\
\cline{1-1}
\multicolumn{1}{|c|}{$t_r$} & \\
\cline{1-1}
& \\
\cline{2-2}
& \multicolumn{1}{|c|}{$z_a$} \\
\cline{2-2}
&
\end{tabular} \qquad\qquad\qquad
\begin{tabular}{|c|c|}
& \\
& \\
\cline{1-2}
$t_r$ & $z_a$ \\
\cline{1-2}
& \\
&
\end{tabular}
\qquad\qquad\qquad
\begin{tabular}{|c|c|}
& \ \ \  \\
\cline{1-1}
$t_a$ & \\
\cline{1-1}
& \\
\cline{1-1}
$t_b$ & \\
\cline{1-1}
& \\
\end{tabular}
\]

\[(t_r-z_a+y) \qquad\qquad (t_r-z_a) \qquad\qquad
(\lambda+t_r-z_a-wy) \qquad\qquad
\frac{(t_a-t_b+y)}{(t_a-t_b)}\]
\[type-1 \qquad\qquad\ type-2 \qquad\qquad\qquad type-3\qquad\qquad\qquad type-4\]
For each variable $t_i$ in the table consider all these factors and multiply them together. This is ``the term associated with the filled table''.

One sees that $W_I/y^k$ is the sum of terms associated with the filled tables corresponding to all choices of $\sigma\in S_k$. For example, if $n=3$ then $W_{\{2,3\}}$ is $y^2$ times the sum of two terms associated with the filled tables
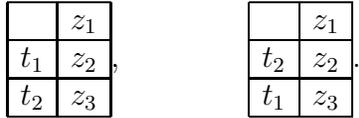
\[
\begin{tabular}{|c|c|}
\hline
  & $z_1$ \\
\hline
 $t_1$ & $z_2$ \\
\hline
$t_2$ & $z_3$ \\
\hline
\end{tabular},\qquad\qquad
\begin{tabular}{|c|c|}
\hline
 & $z_1$ \\
\hline
 $t_2$ & $z_2$ \\
\hline
$t_1$ & $z_3$ \\
\hline
\end{tabular}.
\]
The term corresponding to the first filled table is
\[
\underbrace{(t_1-z_1+y)(t_2-z_1+y)(t_2-z_2+y)}_{type-1}
\underbrace{(t_1-z_3)}_{type-2}
\underbrace{(\lambda+t_1-z_2-y)(\lambda+t_2-z_3)}_{type-3}
\underbrace{\frac{t_1-t_2+y}{t_1-t_2}}_{type-4},
\]
and the term corresponding to the second one is obtained by replacing $t_1$ and $t_2$.

In the next section we will substitute $z_a$'s into the $t_i$ variables
according to some rules. Thus we obtain terms corresponding to tables filled
with only $z_a$ variables (no $t_i$'s). If in such a substitution we have
a filled table containing
\begin{tabular}{c|c}
& \\
\cline{1-1}
\multicolumn{1}{|c|}{$z_a$} & \\
\cline{1-1}
& \\
\cline{2-2}
& \multicolumn{1}{|c|}{$z_a$} \\
\cline{2-2}
&
\end{tabular},
then the term corresponding to that table is 0. This phenomenon is behind the substitution lemmas of the next section.

\section{Restrictions of weight functions to fixed points}
\label{sec:interpolation}

\subsection{Interpolation properties} 

For  $I\subset [n]$ denote $z_I=\{z_i\ |\  i\in I\}$.
As before let $k \leq n$, $\sigma\in S_n$, $I,J  \in \Ik$.

\begin{lemma} \label{lem:divbye}
The polynomial $W_{\sigma,I}(\lambda,\zz_J,\zz,y)$ is divisible by $e_k(\zz_J,y)$ for all $\sigma, I,J$.
\qed
\end{lemma}

In other words, although $\Wt_{\sigma,I}$ (see Section \ref{sec:Wversions}) are rational functions in the $\tt$ variables, all their $\zz_J$ substitutions are polynomials.

\begin{lemma} \label{lem:zero}
$\Wt_{\sigma,I}(\lambda,\zz_J,\zz,y)=0$ unless $J \leq_\sigma I$.
\qed
\end{lemma}

\begin{lemma} \label{lem:div}
$\Wt_{\sigma,I}(\lambda,\zz_J,\zz,y)$ is divisible by $e^{\ver}_{\sigma,J,-}$ for all $\sigma,I,J$.
\qed
\end{lemma}

\begin{lemma} \label{lem:main}
\[
\Wt_{\sigma,I}(\lambda,\zz_I,\zz,y)= 
(-1)^{(n+1)k+\ell_{\sigma,I}} C^{(0)}_{\sigma,I} \cdot e^{\hor}_{\sigma,I,-} e^{\ver}_{\sigma,I,-}.
\]
\qed
\end{lemma}

\begin{lemma} \label{lem:small}
$\Wt_{\sigma,I}(\lambda,\zz_J,\zz,y)$ is a polynomial of degree $k(n-k)+k$, where $\deg \lambda=\deg z_i =\deg y=1$.
Its $\lambda$-degree is at most $k$.
It is  divisible by $y$ if $J\not= I$.
\qed

\end{lemma}

Lemmas \ref{lem:divbye}-\ref {lem:small} are analogs  of lemmas in \cite[Section 6.1]{RTV2} and proved
similarly. The lemmas follow from the diagrammatic presentation of weight functions in Section 
\ref{combI}.

\subsection{Cohomology classes $\kappa_{\sigma,I}$}
 \label{sec:existence}

\begin{lemma} \label{thm:kappaW}
There exist elements $\kappa_{\sigma,I}$ in $H_T^*(\TGrkn)\ox \L_{\lambda,\zz,y}$ such that for all torus fixed points $x_J$ in $\TGrkn$ we have
\[
\kappa_{\sigma,I}|_{x_J}=\Wt^+_{\sigma,I}(\lambda,\zz_J,\zz,y).
\]
\end{lemma}

\begin{proof}
The substitutions $\Wt^+_{\sigma,I}(\lambda,\zz_J,\zz,y)$ belong to $\C[\zz,y]\otimes \L_{\lambda,\zz,y}$ by Lemma \ref{lem:divbye} and 
they obviously satisfy the divisibility properties of Section \ref{sec:loc}. This proves the statement. 
\end{proof}

We can informally write that $\kappa_{\sigma,I}=[\Wt^+_{\sigma,I}]$. Lemmas \ref{lem:zero}-\ref{lem:small}
give  properties of  restrictions of classes $\kappa_{\sigma,I}$ at the fixed points.

\begin{remark} 
\label{rem CSM} Each class $\kappa_{\sigma,I}$ is a polynomial in $\la$ of degree $k$. The coefficient
of $\la^k$ is the class  denoted by $\kappa_{\sigma,I}$ in \cite{RTV1}. The classes $\kappa_{\sigma,I}$ of \cite{RTV1} are
the Chern-Schwartz-MacPherson classes of Schubert cells, see \cite{RV}. In that sense the classes 
$\kappa_{\sigma,I}$  of Lemma \ref{thm:kappaW} are  dynamical deformations of 
the Chern-Schwartz-MacPherson classes of Schubert cells.

\end{remark}

\section{Stable envelope maps and R-matrices}
\label{Stable envelope maps and R-matrices}

\subsection{Stable envelope maps} \label{sec:stab}

Recall from (\ref{eqn:loc2}) that $H^*_T(\XX_n) \ox \L_{\lambda,\zz,y}$ is a free module over $\L_{\lambda,\zz,y}$. For a fixed $\sigma \in S_n$ the elements $\kappa_{\sigma,I}$ form a basis in it, because of Lemma~\ref{thm:kappaW} and the triangularity property of Lemma~\ref{lem:zero}.

\begin{definition} For $\sigma\in S_n$ we define the {\em stable envelope map}
\[
\Stab_{\sigma}: H^*_{T}(\XX_n^T) \ox \L_{\lambda,\zz,y} \to H^*_T(\XX_n) \ox \L_{\lambda,\zz,y},
\qquad\qquad
1_I \mapsto \kappa_{\sigma,I},
\]
where $I\in \Ik$ for $0 \leq k \leq n$.
\end{definition}

The $\Stab_\sigma$ map is an isomorphism of free $\L_{\lambda,\zz,y}$-modules.

Recall that we identified $H^*_T(\XX_n^T)\ox \L_{\lambda,\zz,y}$ with $(\C^2)^{\ox n}\ox \L_{\lambda,\zz,y}$ via $1_I\mapsto v_I$ and we identified $\kappa_{\sigma,I}$ with $[\Wt^+_{\sigma,I}]$. Hence $\Stab_{\sigma,I}$ can also be viewed as a map
\[
\Stab_\sigma : (\C^2)^{\ox n}\ox \L_{\lambda,\zz,y} \to H^*_T(\XX_n) \ox \L_{\lambda,\zz,y}, \qquad\qquad v_I \mapsto [\Wt^+_{\sigma,I}].
\]
For example, for $n=2$, $k=1$ we have 
\begin{align*}
\Stab_{\id}: & & & \\
& v_{\{1\}}=v_1\ox v_2  \mapsto  [\Wt^+_{\id,\{1\}}]=\left[\frac{1}{\lambda(\lambda+y)}
(\lambda+t_1-z_1+y)(t_1-z_2)\right] \\
& v_{\{2\}}=v_2\ox v_1  \mapsto  [\Wt^+_{\id,\{2\}}]=\left[\frac{1}{(\lambda-y)\lambda}
(t_1-z_1+y)(\lambda+t_1-z_2)\right]\\
\Stab_{s_{1,2}}: & & & \\
& v_{\{1\}}=v_1\ox v_2  \mapsto  [\Wt^+_{s_{1,2},\{1\}}]=\left[\frac1{(\lambda-y)\lambda}
(\lambda+t_1-z_1)(t_1-z_2+y)\right] \\
& v_{\{2\}}=v_2\ox v_1  \mapsto  [\Wt^+_{s_{1,2},\{2\}}]=\left[\frac1{\lambda(\lambda+y)}
(t_1-z_1)(\lambda+t_1-z_2+y)\right].
\end{align*}

\subsection{Geometric R-matrices}

For $\sigma, \sigma' \in S_n$ we define the {\em geometric R-matrix}
\[
\R_{\sigma',\sigma}=\Stab_{\sigma'}^{-1} \circ \Stab_{\sigma} \in
\End\left( H_T^*(\XX_n^T) \ox \L_{\lambda,\zz,y} \right) =
\End\left( (\C^2)^{\ox n} \ox \L_{\lambda,\zz,y} \right).
\]

\noindent For example, the calculation above yields that for $n=2$ the matrix of $\R_{\id, s}$ in the basis
\[v_1 \ox v_1,\quad  v_1 \ox v_2, \quad  v_2 \ox v_1,\quad  v_2 \ox v_2\]
is
\[
\left( \begin{array}{cccc}
1 & 0 & 0 & 0\\
0 & \frac{(\lambda + y)(z_1-z_2)}{\lambda(z_1-z_2-y)} &  -\frac{(\lambda + z_1-z_2)y}{\lambda(z_1-z_2-y)} &0
\\
0 & -\frac{(\lambda - z_1+z_2)y}{\lambda(z_1-z_2-y)}  & \frac{(\lambda - y)(z_1-z_2)}{\lambda(z_1-z_2-y)}&0
\\
0 & 0 & 0 & 1
\end{array} \right).
\]

\subsection{Dynamical R-matrix} \label{sec:dynR}

Let $\lambda,z,y$ be parameters. Consider the {\em rational  dynamical R-matrix} $R(\lambda,z,y)$ $\in \End(\C^2 \ox \C^2) \otimes \L_{\lambda,\zz,y} $ given by the formula
\bean
\label{R-m}
R(\la,z,y)=
\left( \begin{array}{cccc}
1 & 0 & 0 & 0\\
0 & \frac{(\lambda + y)z}{\lambda(z-y)} &  -\frac{(\lambda + z)y}{\lambda(z-y)} &0
\\
0 & -\frac{(\lambda - z)y}{\lambda(z-y)}  & \frac{(\lambda - y)z}{\lambda(z-y)}&0
\\
0 & 0 & 0 & 1
\end{array} \right).
\eean
in the basis  $v_1\ox v_1, v_1\ox v_2, v_2\ox v_1, v_2\ox v_2$.

Let $R^{(i,j)}(\lambda,z,y)$ be the linear map that acts on $(\C^2)^{\ox n}\otimes \L_{\lambda,\zz,y} $ in such a way that the dynamical R-matrix acts in the $i$ and $j$ factors (here $(i,j)$ is an ordered pair). For $n=2$ we have $R^{(1,2)}(\lambda,z,y)=R(\lambda,z,y)$.
Easy calculation shows the {\em inversion relation}
\bean
\label{PRP}
R^{(1,2)}(\lambda,z,y)R^{(2,1)}(\lambda,-z,y) = \on{Id}.
\eean

Set $h(v_1)=1$ and $h(v_2)=-1$. Define $h^{(j)}(v_{i_1} \ox  \ldots \ox v_{i_n})=h(v_{i_j})$ for an elementary tensor $v_{i_1} \ox  \ldots \ox v_{i_n}$ and extend this linearly to the tensor product. This notation will be used in the whole paper.

One can verify by direct calculation the {\em dynamical Yang-Baxter equation}
\bean
\label{ybe}
&&
R^{(1,2)}(\lambda-y h^{(3)},z-w)R^{(1,3)}(\lambda,z)R^{(2,3)}(\lambda-y h^{(1)},w) =\phantom{aaaaaaaaaaaaaaa}
\\
&&
\notag
\phantom{aaaaaaaaaaaaaaa} = R^{(2,3)}(\lambda,w)R^{(1,3)}(\lambda-y h^{(2)},z)R^{(1,2)}(\lambda,z-w) .
\eean
Here $R^{(1,2)}(\la - yh^{(3)},z-w)$ means that if
$a\ox b\ox c\in \C^2\ox \C^2\ox \C^2$ and 
$h c = \mu c$, $\mu \in\C$, then 
$R^{(1,2)}(\la- yh^{(3)}, w-z) a\ox b\ox c = R^{(1,2)}(\la- y\mu, w-z)(a\ox b)\ox c$, and the other symbols have a similar meaning.

\subsection{Geometric and dynamical R-matrices coincide}

The calculation at the end of Section~\ref{sec:stab} can be rephrased to the fact that for $n=2$ we have
\[ \R_{\id, s} = R^{(1,2)}(\lambda,z_1-z_2,y), \]
or equivalently
\[
\R_{s, \id} = R^{(2,1)}(\lambda,z_2-z_1,y)=
R^{(1,2)}(\lambda,z_1-z_2,y)^{-1}.
\]
More generally the following proposition holds.

\begin{proposition}[Geometric and dynamical R-matrices coincide]
We have
\[
\R_{\sigma s_{a,a+1},\sigma} = R^{(\sigma(a+1),\sigma(a))} (\lambda-y\sum_{i=a+2}^n h^{(\sigma(i))}, z_{\sigma(a+1)}-z_{\sigma(a)},y)
\]
for any $a=1,\dots,n-1$.
 
\end{proposition}

\begin{proof}
The proposition follows from Proposition \ref{prop:Wrec}, cf. \cite[Theorem 3.7]{RTV1},  \cite[Theorem 7.1]{RTV2}.
\end{proof}

\section{The $\xi_I$ vectors}
\label{sec:xi}

\subsection{Definition}
The  vectors $\xi_I$ defined in this section are  dynamical analogs of the vectors  $\xi_I$  in \cite{GRTV}, \cite{RTV1}, \cite{RTV2}.

\medskip

Recall that for $I\in \Ik$ we have $v_I=v_{i_1}\ox \ldots \ox v_{i_n} \in (\C^2)^{\ox n}$ where $i_j=1$ if $j\in I$ and $i_j=2$ if $j\in \Ibar$.

\begin{definition} For $\sigma\in S_n$, $I\in \Ik$ let
\bean
\label{XI}
\xi_I= \frac{1}{Q_I(\zz,y)} \sum_{J\in \Ik} \Wt^-_J(\lambda,\zz_I,\zz,y) v_{J}.
\eean
\end{definition}

\begin{example} For $n=2$, $k=1$ we have
\[
\xi_{\{1\}}= \lambda v_{\{1\}},\qquad
\xi_{\{2\}}= \frac{ -(\lambda+z_1-z_2)y}{z_1-z_2-y} v_{\{1\}} + \frac{(\lambda-y)(z_1-z_2)}{z_1-z_2-y} v_{\{2\}}.
\]
For $n=3,$ $ k=1$ we have
\bea
&&
\xi_{\{1\}} = (\lambda+y) v_{\{1\}},
\\
&&
\xi_{\{2\}} = \frac{-(\lambda+z_1-z_2+y)y}{z_1-z_2-y} v_{\{1\}} + \frac{\lambda(z_1-z_2)}{z_1-z_2-y} v_{\{2\}},
\eea
\bea
&&
\!\!\!\!\!
\xi_{\{3\}} = \frac{-(\lambda+z_1-z_3+y)y}{z_1-z_3-y} v_{\{1\}}  +
      \frac{-(\lambda+z_2-z_3)(z_1-z_3)y}{(z_1-z_3-y)(z_2-z_3-y)} v_{\{2\}} +
\\
&&
\phantom{aaaaaaaaaaaaaaaaa}
\phantom{aaaaaaaaaaaaaaaaa}
+
      (\lambda-y)\frac{(z_1-z_3)(z_2-z_3)}{(z_1-z_3-y)(z_2-z_3-y)} v_{\{3\}}.
\eea
\end{example}

\noindent Let $\Imin=\{1,\ldots,k\}\subset [n]$ and $\Imax=\{n-k+1,\ldots,n\}\subset [n]$.

\begin{proposition} \label{thm:xiI}
The coefficient of $v_J$ in $\xi_I$ is 0 unless $J\leq_{\id} I$. The coefficient of $v_I$ in $\xi_I$ is
\[
C^{(1)}_{\id,I} \cdot
\prod_{b<a}
\prod_{\satop{a\in I}{b\in \Ibar}} \frac{ z_b-z_a}{z_b-z_a-y}.
\]
In particular,
\[
\xi_{\Imin}=\prod_{i=1}^k (\lambda+(n-k-i)y) \cdot v_{\Imin}.
\]
Also, the coefficient of $v_{\Imax}$ in $\xi_{\Imax}$ is
\[
\prod_{i=1}^k (\lambda-iy) \cdot \prod_{\satop{a\in \Imax}{b\not\in \Imax}} \frac{ z_b-z_a }{z_b-z_a-y}=
\prod_{i=1}^k (\lambda-iy) \cdot \frac{ R_{\Imax}(\zz)}{Q_{\Imax}(\zz)}.
\]
\end{proposition}

\begin{proof}
The statements follow from the definition of $\xi_I$ and the interpolation properties of weight functions in Section \ref{sec:interpolation},
cf. \cite[Propositions 2.14]{GRTV},   \cite[Theorem 8.2]{RTV2}.
\end{proof}

\begin{corollary}
\label{cor basis}
The vectors $\{\xi_I\}_{I\in \mathcal I_k}$  form a basis of the $\L_{\la,\zz,y}$-module $(\C^2)^{\ox n} \ox \L_{\la,\zz,y}$.
\end{corollary}

\subsection{Recursive properties of the $\xi_I$ vectors}

Let $P^{(i,i+1)}$ be the operator switching the $i$th and $i+1$st factors in $(\C^2)^{\ox n}$. Let $K^{(i,i+1)}: f(z_i,z_{i+1}) \mapsto f(z_{i+1},z_i)$ be the operator that replaces the variables $z_i$ and $z_{i+1}$. Here $f$ may also depend on other variables $\lambda,z_1,\ldots,z_{i-1},z_{i+2},\ldots, z_n,y$, and can be vector-valued as well.

Recall the dynamical R-matrix and the notation $h^{(j)}$ from Section \ref{sec:dynR}. Define the operator
\[
\ts_i=R^{(i,i+1)}(\la-y\sum_{k=i+2}^{n} h^{(k)},z_{i}-z_{i+1}) \circ P^{(i,i+1)} \circ K^{(i,i+1)}
\]
in $\End( (\C^2)^{\ox n} ) \ox \L_{\lambda,\zz,y}$.

\begin{proposition}
The $\ts_i$ operators satisfy the relations
\[
\ts_i^2 =1, \qquad  \ts_{i+1}\ts_i \ts_{i+1} = \ts_i \ts_{i+1} \ts_i, \qquad
\ts_i \ts_j = \ts_j \ts_i \ \ \text{if}\ \  |i-j|> 1,
\]
and hence they define an action of $S_n$. Moreover,
\[
\ts_i z_i =z_{i+1} \ts_i, \ \ \ \ \ts_i z_{i+1} =z_i \ts_i,\ \ \ \ \ts_i z_j =z_j \ts_i, \ \ \text{if}\ \  j\not=i,i+1,
\]
where $z_1, \ldots, z_n$ are considered as the scalar operators on $(\C^2)^{\ox n} \ox \L_{\lambda,\zz,y}$ of multiplication by the respective variable.
\end{proposition}
\begin{proof}
The proposition follows from the dynamical Yang-Baxter equation \Ref{ybe} and inversion relation  \Ref{PRP}.
\end{proof}

\begin{proposition}
\label{thm:xi_recursion}
We have $\xi_{s_i(I)}= \ts_i \xi_I$.
\end{proposition}

\begin{proof}
The weight functions satisfy the recursion in Proposition \ref{prop:Wrec}. The $\xi_I$ vectors are defined in terms of the weight functions, and hence they also satisfy the appropriate recursion, cf. \cite[Propositions 2.14]{GRTV},   \cite[Theorem 8.2]{RTV2}.
\end{proof}

Proposition  \ref{thm:xi_recursion} together with the explicit formula for $\xi_{\Imin}$ in Proposition \ref{thm:xiI} could serve as an alternative definition of the $\xi_I$ vectors.

\subsection{$\ts_i$ invariant vectors: components in the $\xi_I$ basis}

\begin{proposition}
\label{prop inv vectors}
The vector-valued function $\zeta=\sum_{I\in \Ik} f_I(\lambda,\zz,y) \xi_I$ is invariant under the $S_n$-action generated by the $\ts_i$ operators if and only if
\[
f_{\sigma(I)}(\lambda,\zz,y) = f_I(\lambda,\zz_{\sigma},y)
\]
for all $\sigma\in S_n$ and $I\in \Ik$.
\end{proposition}

\begin{proof}
The proposition follows from Proposition \ref{thm:xi_recursion}.
\end{proof}

This proposition is a dynamical analog of    \cite[Proposition 9.3]{RTV2}.

\subsection{$\ts_i$ invariant vectors: components in the $v_I$ basis}

\begin{definition}
For a function in $\mu,\zz,y$ define
\[
\hat{s}_{i,\mu}(f)=
\frac{(\mu+z_{i+1}-z_i)y}{(\mu-y)(z_{i+1}-z_i)} f +
\frac{\mu(z_{i+1}-z_i-y)}{(\mu-y)(z_{i+1}-z_i)} K^{(i,i+1)}f
\]
\[
=\frac{\mu+y}{\mu-y}f + \frac{\mu(z_{i+1}-z_i-y)}{\mu-y} \partial_i f,\ \ \ \ \ \ \ \ \ \ \ \ \ \
\]
where $\partial_i f=(f-K^{(i,i+1)}f)/(z_i-z_{i+1})$ is the standard divided difference operator.
\end{definition}

\begin{remark}
Calculation shows that the operator $\hat{s}=\hat{s}_{i,\mu}$ satisfies the identity
\[
(\hat{s}+1)\left(\hat{s}-\frac{\mu+y}{\mu-y}\right)=0.
\]
\end{remark}

\begin{lemma}
\label{prop inv vectors}
The vector-valued function $\zeta=\sum_{I\in \Ik} f_I(\lambda,\zz,y)v_I$ is invariant under the operator $\ts_j$ if and only if
\begin{itemize}
\item $f_I=K_j f_I$ for $j,j+1\in I$ or $j,j+1\not\in I$;
\item $f_{s_j(I)}=\hat{s}_{j,\la-y\nu} f_I$ for $j\not\in I, j+1\in I$ and $\sum_{k=j+2}^n h^{(k)}v_I=\nu v_I$;
\item $f_{s_j(I)}=\hat{s}^{-1}_{j,\la-y\nu} f_I$ for $j\in I, j+1\not\in I$ and $\sum_{k=j+2}^n h^{(k)}v_I=\nu v_I$.
\end{itemize}
\end{lemma}
\begin{proof}
The statement follows from Proposition \ref{thm:xi_recursion}, cf.  \cite[Lemma 9.2]{RTV2}.
\end{proof}

\section{Inverse of the $\Stab_{\id}$ map}
\label{sec:inv}

Define the homomorphism of $\L_{\lambda,\zz,y}$-modules
\bean
\label{nu}
\nu: H^*_T(\XX_n) \ox \L_{\lambda,\zz,y} \to (\C^2)^{\ox n} \ox \L_{\lambda,\zz,y}
\eean
 by
\[
[f(\lambda,\Gamma,\zz,y)] \mapsto \sum_{I\in \Ik}
   \frac{f(\lambda,\zz_I,\zz,y)}{R_I(\zz)} \xi_I, \qquad \text{for}\quad
[f(\lambda,\Gamma,\zz,y)]\in H^*_T(\TGrkn)\ox\L_{\lambda,\zz,y}.
\]

\begin{theorem}
The homomorphisms $\Stab_{\id}$  and $\nu$ are inverse to each other.
\end{theorem}

\begin{proof}
For $K\in \Ik$ we have $\Stab_{\id}(v_K)=\kappa_{\id,K}$ which is equal to $[\Wt^+_K(\lambda,\Gamma,\zz,y)]$.
Then
\[
\nu(\Stab_{\id}(v_K)) = \sum_{I\in \Ik} \frac{ \Wt^+_K(\la,\zz_I,\zz,y) }{R_I(\zz)}\xi_I =
\sum_{I\in \Ik}
\frac{ \Wt^+_K(\la,\zz_I,\zz,y) \sum_{J \in \Ik} \Wt^-_J(\la,\zz_I,\zz,y) v_J}
{R_I(\zz)Q_I(\zz,y)}=
\]
\[
\sum_{J\in \Ik} v_J \left( \sum_{I\in \Ik}
\frac{\Wt^+_K(\la,\zz_I,\zz,y)\Wt^-_J(\la,\zz_I,\zz,y)}{R_I(\zz)Q_I(\zz,y)} \right) = v_K,
\]
where the last equality holds because of the orthogonality Corollary  \ref{cor:ortho}.  Cf. \cite[Lemma 6.7]{RTV1} and
Theorem \cite[Lemma 8.5]{RTV2}.
\end{proof}

\begin{corollary}
We have
\[
\Stab_{\id} (\xi_I) = \left[\prod_{i=1}^k \prod_{j\in \Ibar} (\gamma_{1,i}-z_j)\right].
\]
\end{corollary}

\section{Rational dynamical quantum group $\es$}
\label{sec dyn g}

Definitions and formulas below for the rational dynamical quantum group $\es$ are
semi-classical   $\tau\to i\infty$ limits
of the analogous definitions and formulas 
for the elliptic quantum group $E_{\tau,\,y/2}(\slt)$, see \cite{FV1}-\cite{FV3}.

\subsection{Preliminaries}
\label{Prel}
Recall that  the space $(\C^2)^{\ox n}$ has basis of vectors $v_I$, where $I\subset [n]$ and
\[
\label{v_I}
v_I=v_{i_1}\otimes \ldots \otimes v_{i_n},
\]
with $i_j=1$ if $j\in I$ and $i_j=2$ if $j\in [n]-I$. Denote by $(\C^2)^{\ox n}_k$ the span of $\{v_I \ |\ |I|=k\}$.
We have  $(\C^2)^{\ox n}=\oplus_{k=0}^n(\C^2)^{\ox n}_k$.

\smallskip
Let $\h$ be  commutative one-dimensional Lie algebra with generator $h$.  Define the $\h$-action on $(\C^2)^{\ox n}$ by setting $h=2k-n$
on $(\C^2)^{\ox n}_k$.

\medskip

Now consider $h$ as a variable. Denote by $\L_{\la, yh, y}$ the algebra of rational functions of the form $f/g$, where $f$ is a polynomial
in  $\la, yh, y$ and $g$ is a finite product of factors of the form $\la + kyh + ly$ where $k,l\in\Z$.

\medskip

Recall the dynamical R-matrix $R(\la,w,y)\in\End(\C^2\ox\C^2)$ in \Ref{R-m}.    
 Notice that each entry $R_{ij,kl}(\la,w,y)$ of the matrix $R(\la,w,y)$  has the
expansion
\bea
R_{ij,kl}(\la,w,y) = \sum_{s=0}^\infty R_{ij,kl;s}(\la,y) w^{-s}
\eea
with $R_{ij,kl;s}(\la,,y) \in \L_{\la,y}$.

\subsection{Definition}
\label{Definition}

The {\it rational dynamical quantum group} $\es$ is the unital  algebra  with generators of two types.
The generators of the first type are  functions $f(\la, yh, y)\in \L_{\la,yh,y}$.  The generators of the second type are
elements $L_{ij,s}(\la,y)$, \ $ i,j=1,2$,\ $ s\in{\Z}_{>0}$.

Introduce the generating series
\bea
L_{ij}(\la,w,y) = \sum_{s=0}^\infty L_{ij,s}(\la,y)w^{-s},   \qquad  i,j=1,2,
\eea
and consider them as entries of the $2\times 2$-matrix $L(\la,w,y)=(L_{ij}(\la,w,y))$  called the {\it L-operator}.

Relations in $\es$  involving $f(\la, yh, y)$ are
\begin{eqnarray*}
f(\la, yh, y)g(\la, yh, y) &=& (fg)(\la, yh, y),
\\
f(\la, yh, y)L_{11}(\la,w,y) &=& L_{11}(\la,w,y)f(\la, yh, y), \\
f(\la, yh, y)L_{22}(\la,w,y) &=& L_{22}(\la,w,y)f(\la, yh, y),
\\f(\la, yh, y)L_{12}(\la,w,y) &=& L_{12}(\la,w,y)f(\la, yh-2y, y),
\\
f(\la, yh, y)L_{21}(\la,w,y) &=& L_{21}(\la,w,y)f(\la, yh+2y, y).
\end{eqnarray*}
The relations between the generators of the second type are given by the formula
\bean
\label{es rel}
&&
R^{(12)}(\lambda-y h,w_{12},y)L^{(1)}(\lambda,w_{1},y)L^{(2)}(\lambda-y h^{(1)},w_{2},y) =
\\
&&
\phantom{aaaaaaaaaaaaaaa} = L^{(2)}(\lambda,w_{2},y)L^{(1)}(\lambda-y h^{(2)},w_{1},y)R^{(12)}(\lambda,w_{12},y).
\notag
\eean
Here $h$ is considered as the generator of the  one-dimensional commutative Lie algebra $\h$, see details in \cite{FV1}.
The matrix relation  \Ref{es rel} gives 16 scalar relations.
Here are two of them:
\begin{eqnarray*}
&&
L_{11}(\la,w_1,y)L_{11}(\la-y,w_2,y)
=L_{11}(\la,w_2,y)L_{11}(\la-y,w_1,y),
\\
&&
L_{22}(\la,w_1,y)L_{22}(\la+y,w_2,y)
=L_{22}(\la,w_2,y)L_{22}(\la+y,w_1,y),
\end{eqnarray*}
the other fourteen relations are written down explicitly in \cite[Section 2]{FV1}.

\subsection{$\es$-modules} We define on
$\CCn\ox \L_{\la,(\zz,y)}$ the following $\es$-module structures labeled by elements $\si\in S_n$.

The $\h$-module structure on $\CCn\ox \L_{\la,(\zz,y)}$ does not depend on $\si$ and is defined in Section \ref{Prel}. This
$\h$-module structure induces the action on  $\CCn\ox \L_{\la,(\zz,y)}$ of generators $f(\la, yh,y)\in \L_{\la,hy,y}$.

The $\es$-module structure corresponding to $\si$ has the L-operator
\bea
&&
L(\la,w,y) = R^{(0,1)}(\la-y\sum_{j=2}^{n} h^{(j)}, w-z_{\si (1)},y)R^{(0,2)}(\la-y\sum_{j=3}^{n} h^{(j)}, w-z_{\si (2)},y)\dots
\\
&&
\phantom{aaaaaaaaaaa}
\dots
R^{(0,n-1)}(\la-y h^{(n)}, w-z_{\si(n-1)},y)R^{(0,n)}(\la, w-z_{\si (n)},y).
\eea
We think of $L(\la,w,y)$) as an $2\times2$-matrix with $\End(\CCn)$-valued entires $L_{ij}(\la,w,y)$ depending on
$\la,w,\zz,y$. Expand $L_{ij}(\la,w,y)$ into Laurent series in $w$ at $w=\infty$,
\bean
\label{Lau L_{ij}}
L_{ij}(\la,w,y) = \sum_{s=0}^\infty L_{ij,s}(\la,y) w^{-s}.
\eean
Then $L_{ij,s}(\la,y)\in\End(\CCn)\ox \C[\zz,y]\ox \L_{\la,y}$.

The operators $\{f(\la,yh,y), \ L_{ij,s}(\la,y)\}$ define on $\CCn\ox \L_{\la,(\zz,y)}$ an $\es$-module structure, see
\cite{F1, F2, FV1}.

The space $\CCn\ox\L_{\la,(\zz,y)}$ with the $\es$-module structure corresponding to $\si$
 will be denoted by  $V(z_{\si (1)})\ox\dots\ox V(z_{\si (n)})$ or by $V_\si$ and called the {\it tensor product of evaluation modules}.
Denote
\bea
(V(z_{\si (1)})\ox\dots\ox V(z_{\si (n)}))_k = (\C^2)^{\ox n}_k \ox\L_{\la,(\zz,y)}.
\eea

\begin{example}
The $\es$-action on  $V(z_1)$ is given by the formulas
\begin{align*}
& f(\la,yh,y)v_1= f(\la,y,y)v_1,
&  &
f(\la,yh,y)v_2 = f(\la,-y,y)v_2,
\\
&L_{11}(\la,w,y)v_1 = v_1,
& &L_{11}(\la,w,y)v_2 =  \frac{(\la+ y)(w-z_1)}{\la( w-z_1-y)}v_2,
\\
&L_{12}(\la,w,y)v_1 =  -\frac{ (\la+w-z_1)y}{ \la(w-z_1-y)}v_2,
&&
L_{12}(\la,w,y)v_2 = 0,
\\
&  L_{21}(\la,w,y)v_1 = 0,
&&
L_{21}(\la,w,y)v_2 =  -\frac{ (\la-w+z_1)y}{ \la(w-z_1-y)} v_1,
\\
&L_{22}(\la,w,y)v_1 =  \frac{(\la-y)(w-z_1)}{ \la(w-z_1-y)} v_1,
&&
L_{22}(\la,w,y)v_2 = v_2.
\end{align*}

\end{example}

\begin{example} For example, on  $V(z_1)\ox V(z_2)$  we have
\bea
\label{L_{22} on 11}
L_{22}(\la,w,y) v_1\ox v_1 = \frac{(\la-y)(\la-2y)(w-z_1)(w-z_2)}{\la(\la-y)(w-z_1-y)(w-z_2-y)}v_1\ox v_1,
\eea
\bea
\label{L_{22} on 12}
L_{22}(\la,w,y) v_1\ox v_2 = \frac{\la(w-z_1)}{(\la+y)(w-z_1-y)}v_1\ox v_2,
\eea
\bea
\label{L_{22} on 21}
L_{22}(\la,w,y) v_2\ox v_1 &=& \frac{(\la+y-w+z_1)y}{(\la+y)(w-z_1-y)}
\frac{(\la+w - z_2)y}{\la(w-z_2-y)}
 v_1\ox v_2 +
\\
&+&
 \frac{(\la- y)(w-z_2)}{\la(w-z_2-y)}
v_2\ox v_1.
\notag
\eea
\bea
\label{L_{22} on 22}
L_{22}(\la,w,y) v_2\ox v_2 = v_2\ox v_2.
\eea

\end{example}

\subsection{Operator algebra}

For an $\es$-module $V_\si$, $\si\in S_n$, we  define the {\it operator algebra} $A_\si$ as the unital $\C$-algebra
of the following (difference  in $\la$) operators,
acting on $V_\si$, with generators
$f(\la, yh, y)\in \C(\la; yh, y)$ and $\tilde L_{11}(w,y), \tilde L_{12}(w,y), \tilde L_{21}(w,y),\tilde L_{22}(w,y)$. For \\
$\zeta \in \V$,
we set
\bea
(f(\la, yh, y)\zeta)(\la,\zz,y) &=& f(\la,yh,y)\zeta(\la,\zz,y),
\\
(\tilde L_{11}(w,y)\zeta)(\la,\zz,y) &=& L_{11}(\la,w,y)\zeta(\la-y,\zz,y),
\\
(\tilde L_{21}(w,y)\zeta)(\la,\zz,y) &=& L_{21}(\la,w,y)\zeta(\la-y, \zz,y),
\\
(\tilde L_{12}(w,y)\zeta)(\la,\zz,y) &=& L_{12}(\la,w)\zeta(\la+y,\zz,y),
\\
(\tilde L_{22}(w,y)\zeta)(\la,\zz,y) &=& L_{22}(\la,w)\zeta(\la+y,\zz,y).
\eea
Relations involving $f(\la, y h,y)$ are
\bea
f(\la, y h,y)g(\la,y h,y) &=& (fg)(\la, yh,y),
\\
f(\la-y,y h,y)\tilde L_{11}(w,y) &=& \tilde L_{11}(w,y)f(\la,y h,y),
\\
f(\la+y,y h,y)\tilde L_{22}(w,y) &=& \tilde L_{22}(w,y)f(\la,y h,y),
\\
f(\la+y,yh+2y,y)\tilde L_{12}(w,y) &=& \tilde L_{12}(w,y)f(\la,y h,y),
\\
f(\la-y,y h-2y,y)\tilde L_{21}(w,y) &=& \tilde L_{21}(w,y)f(\la,y h,y).
\eea
The 16 relations between $\tilde L_{11}(w,y), \tilde L_{12}(w,y), \tilde L_{21}(w,y), \tilde L_{22}(w,y)$ are induced  by \Ref{es rel}. Here are two of them:
\bea
\label{ex tilde L}
&&
\tilde L_{11}(w_1,y)\tilde L_{11}(w_2,y)
=\tilde L_{11}(w_2,y)\tilde L_{11}(w_1,y),
\\
&&
\tilde L_{22}(w_1,y)\tilde  L_{22}(w_2,y)
=\tilde  L_{22}(w_2,y)\tilde  L_{22}(w_1,y),
\notag
\eea
the remaining relations are written down explicitly in \cite[Section 3]{FV1}.

\medskip
Let $\dl : \zeta(\la,\zz, y) \mapsto \zeta(\la+y,\zz,y)$ denote the shift operator. Then
\bea
\notag
&
\tilde L_{11}(w,y) = L_{11}(\la,w,y)\dl^{-1},
&\qquad
\tilde L_{21}(w,y) = L_{21}(\la,w,y)\dl^{-1},
\\
\label{}
&
\tilde L_{12}(w,y) = L_{12}(\la,w,y)\dl,
&\qquad
\tilde L_{22}(w,y) = L_{22}(\la,w,y)\dl.
\eea
Each of these difference operators has the expansion of the form
\bean
\label{exp tilde}
\tilde L_{ij}(w,y) = \sum_{s=0}^\infty \tilde L_{ij,s}(y)w^{-s},
\eean
where
\bea
\tilde L_{ij,s}(y) = L_{ij,s}(\la, y) \dl^{\pm1},
\eea
and the sign is plus if $j=1$ and the sign is minus if $j=2$.

\subsection{Isomorphisms of modules $V_\si$}
\label{sec:morphisms}

\begin{lemma}
\label{thm:Risom}
For any $\si\in S_n$ and $i$, $1\le i<n$, the map
\begin{eqnarray*}
\hat R_{i,i+1} &:&
V(z_{\si(1)})\ox\dots\ox  V(z_{\si(i+1)})\ox V(z_{\si(i)})\ox \dots \ox V(z_{\si(n)})
\\
&&
\to
V(z_{\si(1)})\ox\dots\ox  V(z_{\si(i)})\ox V(z_{\si(i+1)})\ox \dots \ox V(z_{(n)}),
\end{eqnarray*}
where
\[
\hat R_{i,i+1} = R^{(i,i+1)}(\la-y\sum_{k=i+2}^{n} h^{(k)},z_{\si(i)}-z_{\si(i+1)}) P^{(i,i+1)},
\]
commutes with the action of the operator algebra.
\end{lemma}

\begin{proof}  The statement follows from the dynamical Yang-Baxter equation.
\end{proof}

\subsection{Determinant}
The element
\bean
\label{DET}
\widetilde{\Det}(w) &=&
\frac{\la}{\la-y h}(\tilde L_{22}(w+y,y)\tilde L_{11}(w,y)
- \tilde L_{12}(w+y,y)\tilde L_{21}(w,y)) =
\\
\notag
&=&
\frac{\la}{\la-y h}(\tilde L_{11}(w+y,y)\tilde L_{22}(w,y)
- \tilde L_{21}(w+y,y)\tilde L_{12}(w,y))
\eean
of the operator algebra is called the {\it determinant element}.

\begin{theorem} [\cite{FV1}]
\label{deT}
The determinant element $\widetilde{\Det}(w,y)$ is a central element of the operator algebra.
It acts on $V_\si$ as multiplication by
\bea
\prod_{i=1}^n \frac{w-z_i+y}{w-z_i}.
\eea
\end{theorem}

\subsection{ Vectors $\xi_I$ as eigenvectors}
\label{sec:xieigen}

We are interested in the action of the generating series $\tilde{L}_{22}(w,y)$ on
$V(z_1)\ox \cdots \ox V(z_n)$.

\begin{proposition}
\label{eigen}
Consider the vector $\xi_I$ defined in \Ref{XI} as an element of
$
V(z_1)\ox \ldots \ox V(z_n).
$
Then
\[ \tilde{L}_{22}(w,y) \xi_I = \prod_{i\in I} \frac{w-z_i}{w-z_i-y} \cdot \xi_I.
\]
\end{proposition}

\begin{proof}
The statement follows by direct calculation for $\Imin$. For the other $\xi_I$ it follows from Proposition \ref{thm:xi_recursion} and Lemma \ref{thm:Risom}.
\end{proof}

By Proposition \ref{eigen},
\bea
\tilde{L}_{22,0}(y) \,\xi_I \,= \, \xi_I .
\eea
  By Corollary \ref{cor basis} the vectors $\xi_I$ form a basis of the $\L_{\la,(\zz,y)}$-module
$V(z_1)\ox \ldots \ox V(z_n)$. For any  $\zeta=\sum_I f_I(\la,\zz, y) \xi_I
\in V(z_1)\ox \ldots \ox V(z_n)$ we have  $\tilde{L}_{22,0}(y) \zeta=\sum_I f_I(\la+y,\zz, y) \xi_I$.
Hence $\tilde{L}_{22,0}(y)$ is invertible and
\bea
\tilde{L}_{22,0}(y)^{-1} \zeta=\sum_I f_I(\la - y,\zz, y) \xi_I.
\eea
The invertibility of $\tilde{L}_{22,0}(y)^{-1}$ allows us to invert the generating series $\tilde{L}_{22}(w,y)$ and define
the generating series $\tilde{L}_{22}(w,y)^{-1}$ such that $\tilde{L}_{22}(w,y)^{-1} \tilde{L}_{22}(w,y) = \Id$.

\subsection{Gelfand-Zetlin algebra}
We define the {\it Gelfand-Zetlin algebra} $\B$ of the $\es$-module $V_{\id}=V(z_1)\ox \ldots \ox V(z_n)$  as the unital commutative
 algebra
 generated by   $\tilde L_{22}(w,y)$,   $\widetilde {Det}(w,y)$, $\tilde{L}_{22,0}(y)^{-1} $, and $\C[y^{\pm1}]$. 
  Here $\C[y^{\pm1}]$ is the algebra of Laurent polynomials in $y$.
More precisely, we expand
\bea
\tilde L_{22}(w,y) =  \sum_{s=0}^\infty \tilde L_{22,s}(y)  w^{-s},
\qquad
\widetilde{Det}(w,u)= \sum_{s=0}^\infty  \widetilde{Det}_s(y) w^{-s},
\eea
and define the Gelfand-Zetlin algebra $\B$ of $V_{\id}$ as the unital commutative algebra
generated by the operators $\tilde L_{22,s}(y),  \widetilde{Det}_s(y) , s\in \Z_{\ge 0}$, $\tilde{L}_{22,0}(y)^{-1} $,
and $\C[y^{\pm1}]$. 

The Gelfand-Zetlin algebra preserves the subspaces $(V(z_1)\ox \ldots \ox V(z_n))_k$. Its restriction to
$(V(z_1)\ox \ldots \ox V(z_n))_k$ will be denoted by $\B_k$.

\subsection{Action of off-diagonal entries of the operator algebra}\label{sec:offdiag}

The off-diagonal elements $\tilde{L}_{1,2}$ and $\tilde{L}_{2,1}$ map different weight subspaces of $V(z_1)\ox \ldots \ox V(z_n)$ into each other.

\begin{proposition}
\label{thm off-diag}

Let $\tilde{F}(w,y)=\tilde{L}_{12}\circ \tilde{L}_{22}^{-1}$ and $\tilde{E}(w,y)=\tilde{L}_{22}^{-1}\circ \tilde{L}_{21}$.
We have
\[
\tilde{F}(w,y)\xi_I =
c_F \sum_{i\in I} \frac{ \xi_{I-\{i\}} }{w-z_i} (\lambda+w-z_i+y(n-2k+1))
\prod_{s\in I-\{i\}} \frac{z_i-z_s-y}{z_i-z_s},
\]
\[
\tilde{E}(w,y)\xi_I =
c_E
\sum_{i\in \Ibar} \frac{ \xi_{I\cup\{i\}} }{w-z_i} (\lambda-w+z_i-y)
\prod_{s\in \Ibar-\{i\}} \frac{z_s-z_i-y}{z_s-z_i},
\]
where $c_F=-y$ and $c_E=-y\frac{1}{(\lambda-y)(\lambda-2y)}$.
\end{proposition}

\begin{proof}
The proof is by straightforward calculation, cf. \cite[Theorem 4.16]{GRTV} and \cite[Theorem 11.8]{RTV2}.
\end{proof}

Notice that the action of $\tilde L_{11}(w,y)$ on the elements $\xi_I$ can be recovered from the actions of $\widetilde{Det}(w,y)$,
$\tilde{L}_{1,2}(w,y)$, $\tilde{L}_{2,1}(w,y)$, $\tilde{L}_{2,2}(w,y)$, see \Ref{DET}.

\section{Action of dynamical quantum group on cohomology}
\label{on coho}

\subsection{Action of Gelfand-Zetlin algebra on cohomology}

Recall  the spaces
\bea
V(z_1)\ox \ldots \ox V(z_n) = (\C^2)^{\ox n} \ox \L_{\la,(\zz,y)},
\qquad
H^*_{T}(\XX_n)\ox \L_{\la,(\zz,y)}.
\eea
Since $V(z_1)\ox \ldots \ox V(z_n)$ is an $\es$-module, the isomorphism
\[
\Stab_{\id} : V(z_1)\ox \ldots \ox V(z_n)  \to H^*_{T}(\XX_n) \ox \L_{\la,(\zz,y)},\qquad
1_I \mapsto \kappa_{\id,I},
\]
of free $\L_{\la,(\zz,y)}$-modules induces on $H^*_{T}(\XX_n) \ox \L_{\la,(\zz,y)}$ a structure on an $\es$-module.
In this section we describe the $\es$-action on $H^*_{T}(\XX_n) \ox \L_{\la,(\zz,y)}$. We start with the action of the Gelfand-Zetlin algebra
$\B$, which preserve the subspaces $(V(z_1)\ox \ldots \ox V(z_n))_k\subset V(z_1)\ox \ldots \ox V(z_n)$.

Recall that the operator  $\delta$ acts on $H^*_T(\TGrkn) \ox \L_{\la,(\zz,y)}$ by the formula
\[
\delta \cdot [f(\lambda,\gamma,\zz,y)] = [f(\lambda+y,\gamma,\zz\,y)]
\qquad \text{for}\ [f(\la,\gamma,\zz,y)]\in H^*_T(\TGrkn) \ox \L_{\la,(\zz,y)}.
\]

\begin{proposition}
\label{Prop}
For any $k$ and $\zeta \in (V(z_1)\ox \ldots \ox V(z_n))_k $ we have
\bea
&&
\Stab_{\id}
\left(
\tilde{L}_{2,2}(w,y)\zeta
\right)
=
\left[
\prod_{i=1}^k \frac{w-\gamma_{1,i}}{w-\gamma_{1,i}-y} \right]\cdot \delta
\cdot
\Stab_{id}(\zeta),
\\
&&
\Stab_{\id}
\left(
\widetilde{Det}(w,y)\zeta
\right)
=
\left[ \prod_{i=1}^n \frac{w-z_i+y}{w-z_i}
\right] \cdot
\Stab_{id}(\zeta).
\eea

\end{proposition}

\begin{proof}
The statement follows from Section \ref{sec:xieigen}.
\end{proof}

Consider the expansions
\bea
\prod_{i=1}^k \frac{w-\gamma_{1,i}}{w-\gamma_{1,i}-y}  = \sum_{s=0}^\infty a_{s}(\ga_{1,1},\dots,\ga_{1,k}, y) w^{-s},
\qquad
\prod_{i=1}^n \frac{w-z_i+y}{w-z_i} = \sum_{s=0}^\infty b_s(\zz,y) w^{-s}.
\eea

\begin{lemma}
\label{lem gen}
The elements $a_{s}(\ga_{1,1},\dots,\ga_{1,k}, y)$, $s\in \Z_{\geq 0}$, generate the $\C[y^{\pm 1}]$-module
 \\
 $\C[\ga_{1,1},\dots,\ga_{1,k}]^{S_k}\ox \C[y^{\pm 1}]$ and
the elements $b_{s}(\zz, y)$, $s\in \Z_{\geq 0}$, generate the $\C[y^{\pm 1}]$-module $\C[\zz]^{S_n}\ox \C[y^{\pm 1}]$.
\qed
\end{lemma}

The subalgebra
\bea
H^*_{GL_n\times\C^\times}(\TGrkn)\ox\C[y^{\pm1}]\subset
H^*_T(\TGrkn) \ox \L_{\la,(\zz,y)}
\eea
acts on $H^*_T(\TGrkn) \ox \L_{\la,(\zz,y)}$ by multiplication.

\begin{theorem}
\label{thm MAIN}
The map
\bean
\label{iSo}
\al :\quad \tilde L_{22,s}(y) \mapsto [a_s(\Ga,y)]\cdot\delta,
\qquad
\widetilde{Det}_{s}(y) \mapsto [b_s(\zz,y)],
\qquad s\in\Z_{\geq 0},
\eean
defines an isomorphism of the Gelfand-Zetlin algebra $\B_k$ acting on
$(V(z_1)\ox \ldots \ox V(z_n))_k $  and the algebra $H^*_{GL_n}(\TGrkn) \ox \C[y^{\pm1}]\ox \C[\delta^{\pm1}]$
acting on $H^*_T(\TGrkn) \ox \L_{\la,(\zz,y)}$.

Consider $H^*_T(\TGrkn) \ox \L_{\la,(\zz,y)}$ as an $H^*_{GL_n\times \C^\times}(\TGrkn) \ox \C[y^{\pm1}]\ox \C[\delta^{\pm1}]$-module.
Then the isomorphisms
\bea
\al  &:& \B_k \to H^*_{GL_n}(\TGrkn) \ox \C[y^{\pm1}]\ox \C[\delta^{\pm1}],
\\
\Stab_{\id}  &:& (V(z_1)\ox \ldots \ox V(z_n))_k  \to H^*_{T}(\TGrkn) \ox \L_{\la,(\zz,y)}
\eea
define an isomorphism of the $\B_k$-module $(V(z_1)\ox \ldots \ox V(z_n))_k$ and the
$H^*_{GL_n\times\C^\times }(\TGrkn) \ox \C[y^{\pm1}]\ox \C[\delta^{\pm1}]$-module
$H^*_T(\TGrkn) \ox \L_{\la,(\zz,y)}$.

\end{theorem}

\begin{proof}
The theorem follows from Proposition \ref{Prop} and Lemma \ref{lem gen}.
\end{proof}

\subsection{Action of $\es$ on cohomology}

The actions of the off-diagonal entries of the operator algebra can be interpreted as actions on the cohomology of the Grassmannians. Namely, for
$$ \Ga=(\gamma_{1,1},\ldots,\gamma_{1,k}; \gamma_{2,1},\ldots, \gamma_{2,n-k})$$
denote
$$ \Ga^{'i}=(\gamma_{1,1},\ldots,\gamma_{1,k-1},\gamma_{2,i}; \gamma_{2,1},\ldots,\check{\gamma}_{2,i},\ldots, \gamma_{2,n-k+1})$$
and
$$
\Ga^{i'}=(\gamma_{1,1},\ldots,\check{\gamma}_{1,i},\ldots,\gamma_{1,k+1}; \gamma_{1,i},\gamma_{2,1},\ldots, \gamma_{2,n-k-1}).$$
Here the $\check{\ }$ sign means omission.

\begin{proposition}
\label{thm:offdiagactionH}
Conjugated by the $\Stab_{\id}$ isomorphism, the series $\tilde{F}(w)$ and $\tilde{E}(w)$  act as
\bea
\tilde{F}(w): H^*_T(T^*\Gr_{k}\C^n)\ox \L_{\lambda,(\zz,y)} \to  H^*_T(T^*\Gr_{k-1}\C^n)\ox \L_{\lambda,(\zz,y)}
\eea
\bea
&&
[f(\lambda,\Ga,\zz,y)] \mapsto \qquad\qquad\qquad \
\\
&&
\phantom{aa}
\left[
(-1)^{k-1}c_F \sum_{i=1}^{n-k+1}
\frac{ f(\lambda,\Gamma^{'i},\zz,y)}{w-\gamma_{2,i}}
(\lambda+w-\gamma_{2,i}+y(n-2k+1))
\frac{ \prod_{j=1}^{k-1}(\gamma_{2,i}-\gamma_{1,j}-y)}
{\prod_{j=1, j\neq i}^{n-k+1} (\gamma_{2,i}-\gamma_{2,j})} \right],
\eea
and
\[
\tilde{E}(w): H^*_T(T^*\Gr_{k}\C^n)\ox \L_{\lambda,(\zz,y)} \to  H^*_T(T^*\Gr_{k+1}\C^n)\ox \L_{\lambda,(\zz,y)}
\]
\bea
&&
[f(\lambda,\Gamma,\zz,y)] \mapsto \qquad\qquad\qquad \
\\
&&
\phantom{aa}
\left[
(-1)^{n-k-1}c_E \sum_{i=1}^{k+1}
\frac{ f(\lambda-2y,\Gamma^{i'},\zz,y)}{w-\gamma_{1,i}}
(\lambda-w+\gamma_{1,i}-y)
\frac{ \prod_{j=1}^{n-k-1}(\gamma_{2,j}-\gamma_{1,i}-y)}
{\prod_{j=1, j\neq i}^{k+1} (\gamma_{1,j}-\gamma_{1,i})} \right],
\eea
respectively.
\end{proposition}

\begin{proof}
The formulas follow from the explicit descriptions of the actions of $\tilde{F}(w)$ and $\tilde{E}(w)$ actions in Section~\ref{sec:offdiag}, as well as the descriptions of $\Stab_{\id}$ and its inverse in Sections \ref{sec:stab} and \ref{sec:inv}.
\end{proof}

The actions of $\tilde L_{11}(w,y)$, $\tilde{L}_{1,2}(w,y)$, $\tilde{L}_{2,1}(w,y)$ on
$H^*_T(\TGrkn) \ox \L_{\la,(\zz,y)}$
  can be recovered from the actions of the Gelfand-Zetlin algebra and the series  $\tilde{F}(w)$, $\tilde{E}(w)$.

\begin{remark}
Similarly to \cite[Appendix]{RSTV}, the formulas of Proposition  \ref{thm:offdiagactionH} can be interpreted as geometric correspondences (pull-back-push-forward maps).
\end{remark}

\subsection{Submodules}
Consider the subspace
\bea
H^*_{GL_n\times\C^\times}(\TGrkn) \ox \L_{\la,y}\ox \C[y^{\pm1}] \subset
H^*_{T}(\TGrkn) \ox \L_{\la,(\zz,y)} .
\eea

\begin{lemma} \label{L2}
The space
\bea
H^*_{GL_n\times\C^\times}(\XX_n) \ox \L_{\la,y}\ox \C[y^{\pm1}] =
\oplus_{k=0}^n H^*_{GL_n\times\C^\times}(\TGrkn) \ox \L_{\la,y}\ox \C[y^{\pm1}]
\eea
is an $\es$-submodule of the $\es$-module $H^*_{T}(\XX_n) \ox \L_{\la,(\zz,y)}$.

\end{lemma}

\begin{proof}
The lemma follows from Proposition \ref{Prop} and Proposition \ref{thm:offdiagactionH}.
\end{proof}

\medskip

Consider the subspace $\mathcal V_k \subset (V(z_1)\ox \ldots \ox V(z_n))_k $ of all elements of the form
\bea
 \sum_{I\in \Ik}
   \frac{f(\la, \zz_I,\zz,y)}{R_I(\zz)} \xi_I, \qquad \text{where}\quad
f(\la, \Gamma,\zz,y)\in \C[\Ga]^{S_k\times S_{n-k}}\ox\C[\zz]^{S_n}\ox \L_{\la,y}\ox \C[y^{\pm1}].
\eea
Set
\bea
\mathcal V=\oplus_{k=0}^n \mathcal V_k \subset V(z_1)\ox \ldots \ox V(z_n).
\eea

\begin{lemma}
\label{L1}
The space $\mathcal V$ is an $\es$-submodule of the $\es$-module $V(z_1)\ox \ldots \ox V(z_n)$.

\end{lemma}

\begin{proof}
The lemma follows from Lemma \ref{L2} and the fact that the isomorphism $\nu$ in \Ref{nu} identifies 
$H^*_{GL_n\times\C^\times}(\TGrkn) \ox \L_{\la,y}\ox \C[y^{\pm1}]$ and $\mathcal V_k$.
\end{proof}

\begin{remark}
Notice that $\mathcal  V$ is invariant with respect to the $S_n$-action
of Proposition \ref {prop inv vectors}. The space $\mathcal V$ is the analog of the spaces 
$\frac 1D\mathcal V^-$ in \cite{GRTV} and \cite{RTV2}.

\end{remark}

\begin{corollary}
\label{cor last}

The isomorphisms $\al $ and $\Stab_{\id}$ of Theorem \ref{thm MAIN} induce an isomorphism of the
$\es$-modules $\mathcal V$ and the $\es$-module $H^*_{GL_n\times\C^\times}(\XX_n) \ox \L_{\la,y}\ox \C[y^{\pm1}] $,
as well as  an isomorphism of the $\B_k$-module $\mathcal V_k$ and the
$H^*_{GL_n\times \C^\times}(\TGrkn) \ox \C[y^{\pm1}]\ox \C[\delta^{\pm1}]$-module
$H^*_{GL_n\times \C^\times}(\TGrkn) \ox \L_{\la,y}\ox\C[y^{\pm1}]$.
\end{corollary}

\begin{proof}
The corollary follows from Theorem \ref{thm MAIN}, Proposition \ref{thm:offdiagactionH}, Lemma \ref{L2},
Lemma \ref{L1}.
\end{proof}

\end{document}